\documentclass[12pt, reqno]{amsart}
\usepackage{amsmath, amsthm, amscd, amsfonts, amssymb, graphicx, color}
\usepackage[bookmarksnumbered, colorlinks, plainpages]{hyperref}

\newtheorem{theorem}{Theorem}[section]

\theoremstyle{definition}
\newtheorem{definition}[theorem]{Definition}

\newtheorem{acknowledgement}[theorem]{Acknowledgement}

\theoremstyle{remark}
\newtheorem{remark}[theorem]{Remark}
\numberwithin{equation}{section}
\begin{document}
\title[A new Generalized fractional derivative and integral]{A new
Generalized fractional derivative and integral}
\author[A. Akkurt, M.E. YILDIRIM and H. YILDIRIM]{Abdullah Akkurt$^{1,}$$%
^{\ast }$, M. Esra YILDIRIM$^{2}$ and H\"{u}seyin YILDIRIM$^{1}$ }
\address{$^{1}$ Department of Mathematics, Faculty of Science and Arts,
University of Kahramanmara\c{s} S\"{u}t\c{c}\"{u} \.{I}mam, 46100,
Kahramanmara\c{s}, Turkey.}
\email{abdullahmat@gmail.com; hyildir@ksu.edu.tr}
\address{$^{2}$ Department of Mathematics, Faculty of Science, University of
Cumhuriyet, 58140, Sivas, Turkey.}
\email{mesra@cumhuriyet.edu.tr}
\keywords{\textbf{ }Fractional Calculus, Fractional derivative,
Conformable Fractional Integral.}
\date{03.03.2017\\
\indent$^{\ast }$ Corresponding author}

\begin{abstract}
In this article, we introduce a new general definition of fractional
derivative and fractional integral, which depends on an unknown kernel. By
using these definitions, we obtain the basic properties of fractional
integral and fractional derivative such as Product Rule, Quotient Rule,
Chain Rule, Roll's Theorem and Mean Value Theorem. We give some examples.
\end{abstract}

\maketitle

\setcounter{page}{1}


\section{Introduction}

The main aim of this paper is to introduced limit definition of the
derivative of a function which obeys classical properties including:
linearity, Product Rule, Quotient Rule, Chain Rule, Rolle's Theorem and Mean
Value Theorem.

Today, there are many fractional integral and fractional derivative
definitions such as Riemann-Liouville, Caputo, Gr\"{u}nwald-Letnikov,
Hadamard, Riesz. For these, please see \cite{Kilbas}, \cite{Katugampola1},
\cite{Samko}. For more information on the Fractional Calculus, please see (%
\cite{Akkurt}, \cite{Abdel}, \cite{iyiola}, \cite{Hammad}, \cite{Hammad1}).

Here, all fractional derivatives do not provide some properties such as
Product Rule, Quotient Rule, Chain Rule, Roll's Theorem and Mean Value
Theorem.

To overcome some of these and other difficulties, Khalil et al. \cite{Khalil}%
, came up with an interesting idea that extends the familiar limit
definition of the derivative of a function given by the following $T_{\alpha
}$%
\begin{equation}
T_{\alpha }\left( f\right) \left( t\right) =\lim_{\varepsilon \rightarrow 0}%
\frac{f\left( t+\varepsilon t^{1-\alpha }\right) -f\left( t\right) }{%
\varepsilon }.  \label{a}
\end{equation}

In \cite{Almeida}, Almeida et al. introduced limit definition of the
derivative of a function as follows,%
\begin{equation}
f^{\left( \alpha \right) }\left( t\right) =\lim_{\varepsilon \rightarrow 0}%
\frac{f\left( t+\varepsilon k\left( t\right) ^{1-\alpha }\right) -f\left(
t\right) }{\varepsilon }.  \label{b}
\end{equation}

Recently, in \cite{Katugampola} Katugampola introduced the idea of
fractional derivative%
\begin{equation}
D_{\alpha }\left( f\right) \left( t\right) =\lim_{\varepsilon \rightarrow 0}%
\frac{f\left( te^{\varepsilon t^{-\alpha }}\right) -f\left( t\right) }{%
\varepsilon }.  \label{c}
\end{equation}

\section{Generalized new fractional derivative}

In this paper, we introduce a new fractional derivative which is generalized
the results obtained in \cite{Almeida}, \cite{Katugampola}, \cite{Khalil}.

In this section we present the definition of the Generalized new fractional
derivative and introduce the Generalized new fractional integral. We
provided representations for the Product Rule, Quotient Rule, Chain Rule,
Roll's Theorem and Mean Value Theorem. Also, we give some applications.

\begin{definition}
\label{d1} Let $k:[a,b]\rightarrow
\mathbb{R}
$ be a continuous nonnegative map such that $k\left( t\right) ,$ $k^{\prime
}\left( t\right) \neq 0,$ whenever $t>a.$ Given a function $%
f:[a,b]\rightarrow
\mathbb{R}
\ $and $\alpha \in \left( 0,1\right) \ $a real, we say that the generalized
fractional derivative\ of $f$ of order $\alpha $ is defined by,%
\begin{equation}
D^{\alpha }\left( f\right) \left( t\right) :=\lim_{\epsilon \rightarrow 0}%
\frac{f\left( t-k\left( t\right) +k\left( t\right) e^{\varepsilon \frac{%
\left( k\left( t\right) \right) ^{-\alpha }}{k^{\prime }\left( t\right) }%
}\right) -f\left( t\right) }{\epsilon }  \label{1.8}
\end{equation}%
exist$.\ $If $f$ is $\alpha -$differentiable in some $\left( 0,a\right) ,\
\alpha >0,\ \lim\limits_{t\rightarrow 0^{+}}f^{\left( \alpha \right) }\left(
t\right) \ $exist, then define%
\begin{equation}
f^{\left( \alpha \right) }\left( 0\right) =\lim\limits_{t\rightarrow
0^{+}}f^{\left( \alpha \right) }\left( t\right) .  \label{1.9}
\end{equation}%
We can write $f^{\left( \alpha \right) }\left( t\right) $ for $D^{\alpha
}\left( f\right) \left( t\right) $ to denote the generalized fractional
derivatives of $f$ of order $\alpha $.
\end{definition}

\begin{remark}
When $k\left( t\right) =t\ $in (\ref{1.8}), it turns out to be the
definition for derivatives of a function,\ in \cite{Katugampola}.
\end{remark}

\begin{remark}
When $\alpha \rightarrow 1\ $and $k\left( t\right) =t\ $in (\ref{1.8}), it
turns out to be the classical definition for derivatives of a function,\ $%
f^{\left( \alpha \right) }\left( t\right) =f^{\prime }\left( t\right) .$
\end{remark}

\begin{theorem}
Let $f:[a,b]\rightarrow
\mathbb{R}
\ $be a$\ $differentiable function and $t>a.\ $Then, $f\ $is a $\alpha -$%
differentiable at $t\ $and%
\begin{equation*}
f^{\left( \alpha \right) }\left( t\right) =\frac{\left( k\left( t\right)
\right) ^{1-\alpha }}{k^{\prime }\left( t\right) }\frac{df}{dt}(t).
\end{equation*}%
Also, if $f^{\prime }$\ is continuous at $t=a,\ $then%
\begin{equation*}
f^{\left( \alpha \right) }\left( a\right) =k^{\prime }\left( a\right) \left(
k\left( a\right) \right) ^{1-\alpha }\frac{df}{dt}(a).
\end{equation*}
\end{theorem}

\begin{proof}
From definition \ref{d1}, we have%
\begin{eqnarray*}
D^{\alpha }\left( f\right) \left( t\right) &=&\lim_{\epsilon \rightarrow 0}%
\frac{f\left( t-k\left( t\right) +k\left( t\right) e^{\varepsilon \frac{%
\left( k\left( t\right) \right) ^{-\alpha }}{k^{\prime }\left( t\right) }%
}\right) -f\left( t\right) }{\epsilon } \\
&& \\
&=&\lim_{\in \rightarrow 0}\frac{\left( t-k\left( t\right) +k\left( t\right) %
\left[ 1+\varepsilon \frac{\left( k\left( t\right) \right) ^{-\alpha }}{%
k^{\prime }\left( t\right) }+\frac{\left( \varepsilon \frac{\left( k\left(
t\right) \right) ^{-\alpha }}{k^{\prime }\left( t\right) }\right) ^{2}}{2!}%
+...\right] \right) -f\left( t\right) }{\varepsilon } \\
&& \\
&=&\lim_{\epsilon \rightarrow 0}\frac{f\left( t+\epsilon \frac{\left(
k\left( t\right) \right) ^{1-\alpha }}{k^{\prime }\left( t\right) }\left[
1+O\left( \epsilon \right) \right] \right) -f\left( t\right) }{\epsilon }.
\end{eqnarray*}%
Taking%
\begin{equation*}
h=\epsilon \frac{\left( k\left( t\right) \right) ^{1-\alpha }}{k^{\prime
}\left( t\right) }\left[ 1+O\left( \epsilon \right) \right]
\end{equation*}%
we have,%
\begin{eqnarray*}
D^{\alpha }\left( f\right) \left( t\right) &=&\lim_{\epsilon \rightarrow 0}%
\frac{f\left( t+h\right) -f\left( t\right) }{\frac{k^{\prime }\left(
t\right) \left( k\left( t\right) \right) ^{\alpha -1}h}{1+O\left( \epsilon
\right) }} \\
&& \\
&=&\frac{\left( k\left( t\right) \right) ^{1-\alpha }}{k^{\prime }\left(
t\right) }\frac{df}{dt}(t).
\end{eqnarray*}
\end{proof}

\begin{theorem}
If a function $f:[a,b]\rightarrow
\mathbb{R}
\ $is $\alpha -$differentiable at $a>0,\ \alpha \in \left( 0,1\right] ,\ $%
then $f$\ is continuous at $a.$
\end{theorem}

\begin{proof}
Since%
\begin{equation*}
f\left( a-k\left( a\right) +k\left( a\right) e^{\varepsilon \frac{\left(
k\left( a\right) \right) ^{-\alpha }}{k^{\prime }\left( a\right) }}\right)
-f\left( a\right) =\tfrac{f\left( a-k\left( a\right) +k\left( a\right)
e^{\varepsilon \frac{\left( k\left( a\right) \right) ^{-\alpha }}{k^{\prime
}\left( a\right) }}\right) -f\left( a\right) }{\epsilon }\epsilon ,
\end{equation*}%
we have%
\begin{equation*}
\lim_{\epsilon \rightarrow 0}\left[ f\left( a-k\left( a\right) +k\left(
a\right) e^{\varepsilon \frac{\left( k\left( a\right) \right) ^{-\alpha }}{%
k^{\prime }\left( a\right) }}\right) -f\left( a\right) \right]
=\lim_{\epsilon \rightarrow 0}\tfrac{\left[ f\left( a-k\left( a\right)
+k\left( a\right) e^{\varepsilon \frac{\left( k\left( a\right) \right)
^{-\alpha }}{k^{\prime }\left( a\right) }}\right) -f\left( a\right) \right]
}{\epsilon }\lim_{\epsilon \rightarrow 0}\epsilon .
\end{equation*}%
Let $h=\epsilon \frac{\left( k\left( t\right) \right) ^{1-\alpha }}{%
k^{\prime }\left( t\right) }\left[ 1+O\left( \epsilon \right) \right] .\ $%
Then,%
\begin{equation*}
\lim_{h\rightarrow 0}\left[ f\left( a+h\right) -f\left( a\right) \right]
=D^{\alpha }\left( f\right) \left( a\right) .0
\end{equation*}%
and%
\begin{equation*}
\lim_{h\rightarrow 0}f\left( a+h\right) =f\left( a\right) .
\end{equation*}%
This completes the proof.
\end{proof}

\begin{theorem}
Let $\alpha \in \left( 0,1\right] $ and $f,g$ be $\alpha -$differentiable at
$a$ point $t>0$. Then,
\end{theorem}

$1.\ D^{\alpha }\left( af+bg\right) \left( t\right) =aD^{\alpha }\left(
f\right) \left( t\right) +bD^{\alpha }\left( g\right) \left( t\right) ,\ $%
for all $a,b\in
\mathbb{R}
\ $(linearity)$.$

\bigskip

$2.D^{\alpha }\left( t^{n}\right) =\frac{\left( k\left( t\right) \right)
^{1-\alpha }}{k^{\prime }\left( t\right) }nt^{n-1}\ $for all $n\in
\mathbb{R}
.$

\bigskip

$3.\ D^{\alpha }\left( c\right) =0,\ $for all constant functions\ $f\left(
t\right) =c.$

\bigskip

$4.\ D^{\alpha }\left( fg\right) \left( t\right) =f\left( t\right) D^{\alpha
}\left( g\right) \left( t\right) +g\left( t\right) D^{\alpha }\left(
f\right) \left( t\right) \ $(Product Rule)$.$

\bigskip

$5.\ D^{\alpha }\left( \dfrac{f}{g}\right) \left( t\right) =\dfrac{f\left(
t\right) D_{\alpha }\left( g\right) \left( t\right) -g\left( t\right)
D_{\alpha }\left( f\right) \left( t\right) }{\left[ g\left( t\right) \right]
^{2}}\ $(Quotient Rule)$.$

\bigskip

$6.\ D^{\alpha }\left( f\circ g\right) \left( t\right) =\frac{\left( k\left(
t\right) \right) ^{1-\alpha }}{k^{\prime }\left( t\right) }f^{\prime }\left(
g\left( t\right) \right) D^{\prime }\left( g\right) \left( t\right) $ (Chain
rule).

\begin{proof}
Part (1) and (3) follow directly from the definition. Let us prove (2), (4),
(5) and (6) respectively. Now, for fixed $\alpha \in \left( 0,1\right] ,\
n\in
\mathbb{R}
\ $and $t>0,\ $we have%
\begin{eqnarray*}
D^{\alpha }\left( t^{n}\right) &=&\lim_{\epsilon \rightarrow 0}\frac{\left(
t-k\left( t\right) +k\left( t\right) e^{\varepsilon \frac{\left( k\left(
t\right) \right) ^{-\alpha }}{k^{\prime }\left( t\right) }}\right) ^{n}-t^{n}%
}{\epsilon } \\
&& \\
&=&\lim_{\epsilon \rightarrow 0}\frac{\left( t+\epsilon \frac{\left( k\left(
t\right) \right) ^{1-\alpha }}{k^{\prime }\left( t\right) }\left[ 1+O\left(
\epsilon \right) \right] \right) ^{n}-t^{n}}{\epsilon } \\
&& \\
&=&\frac{\left( k\left( t\right) \right) ^{1-\alpha }}{k^{\prime }\left(
t\right) }nt^{n-1}.
\end{eqnarray*}%
This completes proof of (2). Then, we shall prove (4). To this end, since $%
f,g$ are $\alpha -$differentiable at $t>0$, note that,%
\begin{eqnarray*}
&&D^{\alpha }\left( fg\right) \left( t\right) \\
&=&\lim_{\epsilon \rightarrow 0}\tfrac{f\left( t-k\left( t\right) +k\left(
t\right) e^{\varepsilon \frac{\left( k\left( t\right) \right) ^{-\alpha }}{%
k^{\prime }\left( t\right) }}\right) g\left( t-k\left( t\right) +k\left(
t\right) e^{\varepsilon \frac{\left( k\left( t\right) \right) ^{-\alpha }}{%
k^{\prime }\left( t\right) }}\right) -f\left( t\right) g\left( t\right) }{%
\epsilon } \\
&& \\
&=&\lim_{\epsilon \rightarrow 0}\left[ \tfrac{f\left( t-k\left( t\right)
+k\left( t\right) e^{\varepsilon \frac{\left( k\left( t\right) \right)
^{-\alpha }}{k^{\prime }\left( t\right) }}\right) g\left( t-k\left( t\right)
+k\left( t\right) e^{\varepsilon \frac{\left( k\left( t\right) \right)
^{-\alpha }}{k^{\prime }\left( t\right) }}\right) -f\left( t\right) g\left(
t-k\left( t\right) +k\left( t\right) e^{\varepsilon \frac{\left( k\left(
t\right) \right) ^{-\alpha }}{k^{\prime }\left( t\right) }}\right) }{%
\epsilon }\right. \\
&& \\
&&+\left. \tfrac{f\left( t\right) g\left( t-k\left( t\right) +k\left(
t\right) e^{\varepsilon \frac{\left( k\left( t\right) \right) ^{-\alpha }}{%
k^{\prime }\left( t\right) }}\right) -f\left( t\right) g\left( t\right) }{%
\epsilon }\right] \\
&& \\
&=&\lim_{\epsilon \rightarrow 0}\left[ \tfrac{f\left( t-k\left( t\right)
+k\left( t\right) e^{\varepsilon \frac{\left( k\left( t\right) \right)
^{-\alpha }}{k^{\prime }\left( t\right) }}\right) -f\left( t\right) }{%
\epsilon }g\left( t-k\left( t\right) +k\left( t\right) e^{\varepsilon \frac{%
\left( k\left( t\right) \right) ^{-\alpha }}{k^{\prime }\left( t\right) }%
}\right) \right] \\
&& \\
&&+f\left( t\right) \lim_{\epsilon \rightarrow 0}\tfrac{g\left( t-k\left(
t\right) +k\left( t\right) e^{\varepsilon \frac{\left( k\left( t\right)
\right) ^{-\alpha }}{k^{\prime }\left( t\right) }}\right) -g\left( t\right)
}{\epsilon } \\
&& \\
&=&D^{\alpha }\left( f\right) \left( t\right) \lim_{\epsilon \rightarrow 0}
\left[ g\left( t-k\left( t\right) +k\left( t\right) e^{\varepsilon \frac{%
\left( k\left( t\right) \right) ^{-\alpha }}{k^{\prime }\left( t\right) }%
}\right) \right] +f\left( t\right) D^{\alpha }\left( g\right) \left( t\right)
\\
&& \\
&=&g\left( t\right) D^{\alpha }\left( f\right) \left( t\right) +f\left(
t\right) D^{\alpha }\left( g\right) \left( t\right) .
\end{eqnarray*}%
Since $g$ is continuous at $t$, $\lim_{\epsilon \rightarrow 0}\left[ g\left(
t-k\left( t\right) +k\left( t\right) e^{\varepsilon k\left( t\right)
^{-\alpha }}\right) \right] =g\left( t\right) .$ This completes the proof of
(4). Next, we prove (5). Similarly,
\begin{eqnarray*}
&&D^{\alpha }\left( \frac{f}{g}\right) \left( t\right) \\
&=&\lim_{\epsilon \rightarrow 0}\frac{\frac{f\left( t-k\left( t\right)
+k\left( t\right) e^{\varepsilon \frac{\left( k\left( t\right) \right)
^{-\alpha }}{k^{\prime }\left( t\right) }}\right) }{g\left( t-k\left(
t\right) +k\left( t\right) e^{\varepsilon \frac{\left( k\left( t\right)
\right) ^{-\alpha }}{k^{\prime }\left( t\right) }}\right) }-\frac{f\left(
t\right) }{g\left( t\right) }}{\epsilon } \\
&& \\
&=&\lim_{\epsilon \rightarrow 0}\tfrac{f\left( t-k\left( t\right) +k\left(
t\right) e^{\varepsilon \frac{\left( k\left( t\right) \right) ^{-\alpha }}{%
k^{\prime }\left( t\right) }}\right) g\left( t\right) -f\left( t\right)
g\left( t-k\left( t\right) +k\left( t\right) e^{\varepsilon \frac{\left(
k\left( t\right) \right) ^{-\alpha }}{k^{\prime }\left( t\right) }}\right) }{%
\epsilon g\left( t-k\left( t\right) +k\left( t\right) e^{\varepsilon \frac{%
\left( k\left( t\right) \right) ^{-\alpha }}{k^{\prime }\left( t\right) }%
}\right) g\left( t\right) } \\
&& \\
&=&\lim_{\epsilon \rightarrow 0}\tfrac{f\left( t-k\left( t\right) +k\left(
t\right) e^{\varepsilon \frac{\left( k\left( t\right) \right) ^{-\alpha }}{%
k^{\prime }\left( t\right) }}\right) g\left( t\right) -f\left( t\right)
g\left( t\right) +f\left( t\right) g\left( t\right) -f\left( t\right)
g\left( t-k\left( t\right) +k\left( t\right) e^{\varepsilon \frac{\left(
k\left( t\right) \right) ^{-\alpha }}{k^{\prime }\left( t\right) }}\right) }{%
\epsilon g\left( t-k\left( t\right) +k\left( t\right) e^{\varepsilon \frac{%
\left( k\left( t\right) \right) ^{-\alpha }}{k^{\prime }\left( t\right) }%
}\right) g\left( t\right) } \\
&& \\
&=&\lim_{\epsilon \rightarrow 0}\frac{1}{g\left( t-k\left( t\right) +k\left(
t\right) e^{\varepsilon \frac{\left( k\left( t\right) \right) ^{-\alpha }}{%
k^{\prime }\left( t\right) }}\right) g\left( t\right) } \\
&& \\
&&\times \left[ \tfrac{f\left( t-k\left( t\right) +k\left( t\right)
e^{\varepsilon \frac{\left( k\left( t\right) \right) ^{-\alpha }}{k^{\prime
}\left( t\right) }}\right) -f\left( t\right) }{\epsilon }g\left( t\right)
-f\left( t\right) \tfrac{g\left( t\right) -g\left( t-k\left( t\right)
+k\left( t\right) e^{\varepsilon \frac{\left( k\left( t\right) \right)
^{-\alpha }}{k^{\prime }\left( t\right) }}\right) }{\epsilon }\right] \\
&& \\
&=&\dfrac{f\left( t\right) D^{\alpha }\left( g\right) \left( t\right)
-g\left( t\right) D^{\alpha }\left( f\right) \left( t\right) }{\left(
g\left( t\right) \right) ^{2}}.
\end{eqnarray*}%
We have implicitly assumed here that $f^{\left( \alpha \right) }$ and $%
g^{\left( \alpha \right) }$ exist and that $g\left( t\right) \neq 0.\ $%
Finally, we prove (6). We have from the definition that%
\begin{eqnarray*}
D^{\alpha }\left( f\circ g\right) \left( t\right) &=&\lim_{\epsilon
\rightarrow 0}\frac{\left( f\circ g\right) \left( t-k\left( t\right)
+k\left( t\right) e^{\varepsilon \frac{\left( k\left( t\right) \right)
^{-\alpha }}{k^{\prime }\left( t\right) }}\right) -\left( f\circ g\right)
\left( t\right) }{\epsilon } \\
&=&\lim_{\epsilon \rightarrow 0}\frac{\left( f\circ g\right) \left(
t+\epsilon \frac{\left( k\left( t\right) \right) ^{1-\alpha }}{k^{\prime
}\left( t\right) }\left[ 1+O\left( \epsilon \right) \right] \right) -\left(
f\circ g\right) \left( t\right) }{\epsilon }.
\end{eqnarray*}%
Let $h=\epsilon \frac{\left( k\left( t\right) \right) ^{1-\alpha }}{%
k^{\prime }\left( t\right) }\left[ 1+O\left( \epsilon \right) \right] \ $%
such that%
\begin{eqnarray*}
D^{\alpha }\left( f\circ g\right) \left( t\right) &=&\lim_{\epsilon
\rightarrow 0}\frac{\left( f\circ g\right) \left( t+\epsilon \frac{\left(
k\left( t\right) \right) ^{1-\alpha }}{k^{\prime }\left( t\right) }\left[
1+O\left( \epsilon \right) \right] \right) -\left( f\circ g\right) \left(
t\right) }{\epsilon } \\
&=&\lim_{h\rightarrow 0}\frac{\left( f\circ g\right) \left( t+h\right)
-\left( f\circ g\right) \left( t\right) }{\frac{k^{\prime }\left( t\right)
\left( k\left( t\right) \right) ^{\alpha -1}h}{1+O\left( \epsilon \right) }}.
\end{eqnarray*}%
Therefore, we have%
\begin{equation*}
D^{\alpha }\left( f\circ g\right) \left( t\right) =\frac{\left( k\left(
t\right) \right) ^{1-\alpha }}{k^{\prime }\left( t\right) }f^{\prime }\left(
g\left( t\right) \right) D^{\prime }\left( g\right) \left( t\right) .
\end{equation*}

This completes the proof of the theorem.
\end{proof}

Now, we will give the derivatives of some special functions.

\begin{theorem}
\label{thm1} Let $a,n\in
\mathbb{R}
\ $and $\alpha \in \left( 0,1\right] .\ $Then we have the following results.
\end{theorem}

$1.\ D^{\alpha }\left( 1\right) =0,$

\bigskip

$2.\ D^{\alpha }\left( e^{ax}\right) =a\frac{\left( k\left( x\right) \right)
^{1-\alpha }}{k^{\prime }\left( x\right) }e^{ax},$

\bigskip

$3.\ D^{\alpha }\left( \sin (ax)\right) =a\frac{\left( k\left( x\right)
\right) ^{1-\alpha }}{k^{\prime }\left( x\right) }\cos (ax),$

\bigskip

$4.\ D^{\alpha }\left( \cos (ax)\right) =-a\frac{\left( k\left( x\right)
\right) ^{1-\alpha }}{k^{\prime }\left( x\right) }\sin (ax),$

\bigskip

$5.\ D^{\alpha }\left( \log _{a}bx\right) =\dfrac{1}{x}\frac{\left( k\left(
x\right) \right) ^{1-\alpha }}{k^{\prime }\left( x\right) }\frac{1}{\ln a},$

\bigskip

$6.\ D^{\alpha }\left( a^{bx}\right) =b\frac{\left( k\left( x\right) \right)
^{1-\alpha }}{k^{\prime }\left( x\right) }a^{bx}\ln a.$

When $\alpha =1\ $and $k\left( t\right) =t\ $in Theorem \ref{thm1}, it turns
out to be the classical derivatives of a function.

\begin{theorem}[Rolle's theorem for $\protect\alpha -$generalized Fractional
Differentiable functions]
\label{thm2} Let $a>0\ $and $f:[a,b]\rightarrow
\mathbb{R}
$ be a function with the properties that,
\end{theorem}

1. $f$ is continuous on $[a,b],$

2. $f$ is a $\alpha $-differentiable{}on $\left( a,b\right) \ $for some $%
\alpha \in \left( 0,1\right) ,$

3. $f(a)=f(b).$

Then, there exist $c\in \left( a,b\right) ,\ $such that $D^{\alpha }\left(
f\right) \left( c\right) =0.$

\begin{proof}
We will prove this theorem by using contradiction. Since $f$ is continuous
on $[a,b]$ and $f(a)=f(b)$, there is $c\in \left( a,b\right) $ at which the
function has a local extrema. Then,%
\begin{equation*}
D^{\alpha }\left( f\right) \left( c\right) =\lim_{\epsilon \rightarrow 0^{-}}%
\tfrac{\left[ f\left( c-k\left( c\right) +k\left( c\right) e^{\varepsilon
\frac{\left( k\left( c\right) \right) ^{-\alpha }}{k^{\prime }\left(
c\right) }}\right) -f\left( c\right) \right] }{\epsilon }=\lim_{\epsilon
\rightarrow 0^{+}}\tfrac{\left[ f\left( c-k\left( c\right) +k\left( c\right)
e^{\varepsilon \frac{\left( k\left( c\right) \right) ^{-\alpha }}{k^{\prime
}\left( c\right) }}\right) -f\left( c\right) \right] }{\epsilon }.
\end{equation*}%
But, the two limits have opposite signs. Hence, $D^{\alpha }\left( f\right)
\left( c\right) =0.$
\end{proof}

When $\alpha =1\ $and $k\left( t\right) =t\ $in Theorem \ref{thm2}, it turns
out to be the classical Rolles's Theorem.

\begin{theorem}[Mean value theorem for Generalized fractional differentiable
functions]
Let $\alpha \in (0,1]$ and $f:[a,b]\rightarrow
\mathbb{R}
$ be a continuous on $[a,b]$ and an $\alpha $-generalized fractional
differentiable mapping on $\left( a,b\right) $ with $0\leq a<b.\ $Let $%
k:[a,b]\rightarrow
\mathbb{R}
$ be a continuous nonnegative map such that $k\left( t\right) ,$ $k^{\prime
}\left( t\right) \neq 0.$ Then, there exists $c\in (a,b)$, such that%
\begin{equation}
D^{\alpha }\left( f\right) \left( c\right) =\frac{f(b)-f(a)}{\frac{k^{\alpha
}\left( b\right) }{\alpha }-\frac{k^{\alpha }\left( a\right) }{\alpha }}.
\label{m1}
\end{equation}
\end{theorem}

\begin{proof}
Let $h$ be a constant. Consider the function,%
\begin{equation}
G\left( x\right) =f\left( x\right) +h\frac{k^{\alpha }\left( x\right) }{%
\alpha }.  \label{m2}
\end{equation}%
$G$ is continuous functions on $[a,b]\ $and integrable $\forall x\in \left(
a,b\right) $. Here, if we choose $G\left( a\right) =G\left( b\right) ,\ $%
then we have%
\begin{equation*}
f\left( a\right) +h\frac{k^{\alpha }\left( a\right) }{\alpha }=f\left(
b\right) +h\frac{k^{\alpha }\left( b\right) }{\alpha }.
\end{equation*}%
Thus,%
\begin{equation}
h=-\frac{f\left( b\right) -f\left( a\right) }{\frac{k^{\alpha }\left(
b\right) }{\alpha }-\frac{k^{\alpha }\left( a\right) }{\alpha }}.  \label{m3}
\end{equation}%
Using (\ref{m3}) in (\ref{m2}), it follows that%
\begin{equation}
G\left( x\right) =f\left( x\right) -\frac{f\left( b\right) -f\left( a\right)
}{\frac{k^{\alpha }\left( b\right) }{\alpha }-\frac{k^{\alpha }\left(
a\right) }{\alpha }}\frac{k^{\alpha }\left( x\right) }{\alpha }.  \label{m4}
\end{equation}%
\begin{eqnarray*}
D^{\alpha }\left( G\right) \left( x\right)  &=&D^{\alpha }\left( f\right)
\left( x\right) -\frac{f\left( b\right) -f\left( a\right) }{\frac{k^{\alpha
}\left( b\right) }{\alpha }-\frac{k^{\alpha }\left( a\right) }{\alpha }}%
D^{\alpha }\left( \frac{k^{\alpha }\left( x\right) }{\alpha }\right)  \\
&=&D^{\alpha }\left( f\right) \left( x\right) -\frac{f\left( b\right)
-f\left( a\right) }{\frac{k^{\alpha }\left( b\right) }{\alpha }-\frac{%
k^{\alpha }\left( a\right) }{\alpha }}\frac{\left( k\left( t\right) \right)
^{1-\alpha }}{k^{\prime }\left( t\right) }\frac{d}{dt}\left( \frac{k^{\alpha
}\left( x\right) }{\alpha }\right)  \\
&=&D^{\alpha }\left( f\right) \left( x\right) -\frac{f\left( b\right)
-f\left( a\right) }{\frac{k^{\alpha }\left( b\right) }{\alpha }-\frac{%
k^{\alpha }\left( a\right) }{\alpha }}.
\end{eqnarray*}%
Then, the function $g$ satisfies the condition of the generalized fractional
Rolle's theorem.\ Hence, there exist $c\in \left( a,b\right) ,$ such that $%
D^{\alpha }\left( G\right) \left( c\right) =0.$\ Using the fact that\ $%
D^{\alpha }\left( \frac{k^{\alpha }\left( x\right) }{\alpha }\right) =1$, we
have%
\begin{equation*}
f^{\left( \alpha \right) }\left( x\right) =\frac{f\left( b\right) -f\left(
a\right) }{\frac{k^{\alpha }\left( b\right) }{\alpha }-\frac{k^{\alpha
}\left( a\right) }{\alpha }}.
\end{equation*}%
Therefore, we get desired result.
\end{proof}

When $\alpha =1\ $and $k\left( t\right) =t\ $in Theorem \ref{thm2}, it turns
out to be the classical Mean Value Theorem.

\section{Generalized new fractional integral}

Now we introduce the generalized fractional integral as follows:

\begin{definition}[Generalized Fractional Integral]
\label{d2} Let $a\geq 0\ $and $t\geq a.\ $Also, let $f$ be a function
defined on $(a,t]$\ and $\alpha \in
\mathbb{R}
.\ $Let $k:[a,b]\rightarrow
\mathbb{R}
$ be a continuous nonnegative map such that $k\left( t\right) ,$ $k^{\prime
}\left( t\right) \neq 0.$\ Then, the $\alpha -$generalized fractional
integral of $f$ is defined by,%
\begin{equation*}
I^{\alpha }\left( f\right) \left( t\right) =\int\limits_{a}^{b}\frac{%
k^{\prime }\left( x\right) f\left( x\right) }{\left( k\left( x\right)
\right) ^{1-\alpha }}dx
\end{equation*}%
if the Riemann improper integral exist.
\end{definition}

\begin{theorem}[Inverse property]
Let $a\geq 0\ $and $\alpha \in (0,1).$Also, let $f$ be a continuous function
such that $I^{a}f$ exist. Let $k:[a,b]\rightarrow
\mathbb{R}
$ be a continuous nonnegative map such that $k\left( t\right) ,$ $k^{\prime
}\left( t\right) \neq 0.\ $Then, for all $t>a,$ we have%
\begin{equation*}
D^{a}\left[ I^{a}f\left( t\right) \right] =f\left( t\right) .
\end{equation*}
\end{theorem}

\begin{proof}
Since $f$ is continuous, then $I^{a}f\left( t\right) $ is clearly
differentiable. Hence,%
\begin{eqnarray*}
D^{a}\left[ I^{a}\left( f\right) \left( t\right) \right] &=&\frac{\left(
k\left( t\right) \right) ^{1-\alpha }}{k^{\prime }\left( t\right) }\frac{d}{%
dt}I^{a}(t) \\
&& \\
&=&\frac{\left( k\left( t\right) \right) ^{1-\alpha }}{k^{\prime }\left(
t\right) }\frac{d}{dt}\int\limits_{a}^{t}\frac{f\left( x\right) k^{\prime
}\left( x\right) }{\left( k\left( x\right) \right) ^{1-\alpha }}dx \\
&& \\
&=&\frac{\left( k\left( t\right) \right) ^{1-\alpha }}{k^{\prime }\left(
t\right) }\frac{f\left( t\right) k^{\prime }\left( t\right) }{\left( k\left(
t\right) \right) ^{1-\alpha }} \\
&& \\
&=&f\left( t\right) .
\end{eqnarray*}
\end{proof}

\begin{theorem}
\label{T2} Let $f:(a,b)\rightarrow
\mathbb{R}
$ be differentiable and $0<\alpha \leq 1$. Let $k:[a,b]\rightarrow
\mathbb{R}
$ be a continuous nonnegative map such that $k\left( t\right) ,$ $k^{\prime
}\left( t\right) \neq 0.\ $Then, for all $t>a$ we have%
\begin{equation}
I^{a}\left[ D^{a}\left( f\right) \left( t\right) \right] =f\left( t\right)
-f\left( a\right) .  \label{1.12}
\end{equation}
\end{theorem}

\begin{proof}
\begin{eqnarray*}
I^{a}\left[ D^{a}\left( f\right) \left( t\right) \right] &=&\int%
\limits_{a}^{t}\frac{k^{\prime }\left( x\right) }{\left( k\left( x\right)
\right) ^{1-\alpha }}D^{a}\left( f\right) (x)dx \\
&& \\
&=&\int\limits_{a}^{t}\frac{k^{\prime }\left( x\right) }{\left( k\left(
x\right) \right) ^{1-\alpha }}\frac{\left( k\left( x\right) \right)
^{1-\alpha }}{k^{\prime }\left( x\right) }\frac{df}{dx}(x)dx \\
&& \\
&=&\int\limits_{a}^{t}\frac{df}{dx}(x)dx \\
&& \\
&=&f\left( t\right) -f\left( a\right) .
\end{eqnarray*}
\end{proof}

\begin{theorem}
(\textbf{Integration by parts}) Let $f,g:[a,b]\rightarrow
\mathbb{R}
$ be two functions such that $fg$ is differentiable. Then%
\begin{equation*}
\int_{a}^{b}f\left( x\right) D^{\alpha }\left( g\right) \left( x\right)
d_{\alpha }x=\left. fg\right\vert _{a}^{b}-\int_{a}^{b}g\left( x\right)
D^{\alpha }\left( f\right) \left( x\right) d_{\alpha }x.
\end{equation*}
\end{theorem}

\begin{proof}
The proof is done in a similar way in \cite{Abdel}.
\end{proof}

\begin{theorem}
\label{T1} Let $f$ and $g$ be functions satisfying the following
\end{theorem}

$\left( a\right) $ continuous on $[a,b],$

$\left( b\right) \ $bounded and integrable functions on $[a,b],$

In addition$,\ $Let $g(x)$ be nonnegative (or nonpositive) on $[a,b]$. Let $%
k:[a,b]\rightarrow
\mathbb{R}
$ be a continuous nonnegative map such that $k\left( t\right) ,$ $k^{\prime
}\left( t\right) \neq 0.$ Let us set $m=\inf \{f(x):x\in \lbrack a,b]\}$ and
$M=\sup \{f(x):x\in \lbrack a,b]\}.\ $Then there exists a number $\xi $ in $%
(a,b)$ such that%
\begin{equation}
\int\limits_{a}^{b}\frac{f\left( x\right) g\left( x\right) k^{\prime
}\left( x\right) }{\left( k\left( x\right) \right) ^{1-\alpha }}dx=\xi
\int\limits_{a}^{b}\frac{g\left( x\right) k^{\prime }\left( x\right) }{%
\left( k\left( x\right) \right) ^{1-\alpha }}dx.  \label{t1}
\end{equation}%
If $f$ continuous on $[a,b],$ then for $\exists x_{0}\in \left[ a,b\right] $%
\begin{equation}
\int\limits_{a}^{b}\frac{f\left( x\right) g\left( x\right) k^{\prime
}\left( x\right) }{\left( k\left( x\right) \right) ^{1-\alpha }}dx=f\left(
x_{0}\right) \int\limits_{a}^{b}\frac{g\left( x\right) k^{\prime }\left(
x\right) }{\left( k\left( x\right) \right) ^{1-\alpha }}dx.  \label{t2}
\end{equation}

\begin{proof}
Let $m=\inf f$, $M=\sup f\ $and $g(x)\geq 0\ $in $[a,b].\ $Then, we get%
\begin{equation}
mg(x)<f(x)g(x)<Mg(x).  \label{t4}
\end{equation}%
Multiplying (\ref{t4}) by $\frac{k^{\prime }\left( x\right) }{\left( k\left(
x\right) \right) ^{1-\alpha }}\ $and integrating (\ref{t4}) with respect to $%
x$ over $(a,b)$, we obtain:%
\begin{equation}
m\int\limits_{a}^{b}\frac{g\left( x\right) k^{\prime }\left( x\right) }{%
\left( k\left( x\right) \right) ^{1-\alpha }}dx<\int\limits_{a}^{b}\frac{%
f(x)g\left( x\right) k^{\prime }\left( x\right) }{\left( k\left( x\right)
\right) ^{1-\alpha }}dx<M\int\limits_{a}^{b}\frac{g\left( x\right)
k^{\prime }\left( x\right) }{\left( k\left( x\right) \right) ^{1-\alpha }}dx.
\label{t5}
\end{equation}%
Then there exists a number $\xi $ in $\left[ m,M\right] $ such that%
\begin{equation*}
\int\limits_{a}^{b}\frac{f(x)g\left( x\right) k^{\prime }\left( x\right) }{%
\left( k\left( x\right) \right) ^{1-\alpha }}dx=\xi \int\limits_{a}^{b}%
\frac{g\left( x\right) k^{\prime }\left( x\right) }{\left( k\left( x\right)
\right) ^{1-\alpha }}dx.
\end{equation*}%
When $g(x)<0$, the proof is done in a similar way.

By the intermediate value theorem, $f$ attains every value of the interval $%
[m,M]$, so for some $x_{0}\ $in$\ [a,b]\ f\left( x_{0}\right) =\xi .\ $Then%
\begin{equation*}
\int\limits_{a}^{b}\frac{f(x)g\left( x\right) k^{\prime }\left( x\right) }{%
\left( k\left( x\right) \right) ^{1-\alpha }}dx=f\left( x_{0}\right)
\int\limits_{a}^{b}\frac{g\left( x\right) k^{\prime }\left( x\right) }{%
\left( k\left( x\right) \right) ^{1-\alpha }}dx.
\end{equation*}%
If\ $g(x)=0,\ $equality (\ref{t1}) becomes obvious; if $g(x)>0,\ $then (\ref%
{t5}) implies%
\begin{equation*}
m<\dfrac{\int\limits_{a}^{b}\frac{f(x)g\left( x\right) k^{\prime }\left(
x\right) }{\left( k\left( x\right) \right) ^{1-\alpha }}dx}{%
\int\limits_{a}^{b}\frac{g\left( x\right) k^{\prime }\left( x\right) }{%
\left( k\left( x\right) \right) ^{1-\alpha }}dx}<M
\end{equation*}%
there exists a point $x_{0}$ in $(a,b)$ such that%
\begin{equation*}
m<f\left( x_{0}\right) <M,
\end{equation*}%
which yields the desired result (\ref{t1}). In particular, when $g(x)=1$, we
get from Theorem \ref{T1} the following result%
\begin{eqnarray*}
\int\limits_{a}^{b}\frac{f\left( x\right) k^{\prime }\left( x\right) }{%
\left( k\left( x\right) \right) ^{1-\alpha }}dx &=&f\left( x_{0}\right)
\int\limits_{a}^{b}\frac{k^{\prime }\left( x\right) }{\left( k\left(
x\right) \right) ^{1-\alpha }}dx \\
&& \\
&=&f\left( x_{0}\right) \left( \frac{k^{\alpha }\left( b\right) }{\alpha }-%
\frac{k^{\alpha }\left( a\right) }{\alpha }\right) .
\end{eqnarray*}%
Thus, we have%
\begin{equation}
f\left( x_{0}\right) =\frac{1}{\frac{k^{\alpha }\left( b\right) }{\alpha }-%
\frac{k^{\alpha }\left( a\right) }{\alpha }}\int\limits_{a}^{b}\frac{%
f\left( x\right) k^{\prime }\left( x\right) }{\left( k\left( x\right)
\right) ^{1-\alpha }}dx.  \label{t10}
\end{equation}%
This (\ref{t10}) is called the mean value or variance of the $f$ function.
\end{proof}

For $\alpha =1$ and $k\left( t\right) =t$ this reduces to the classical mean
value theorem of integral calculus,%
\begin{equation*}
\int\limits_{a}^{b}f\left( x\right) dx=\left( b-a\right) f\left(
x_{0}\right) .
\end{equation*}

\begin{theorem}
Let $a\geq 0\ $and $\alpha \in (0,1].\ $Also, let $f,g:\left[ a,b\right]
\rightarrow
\mathbb{R}
$ be a continuous function. Let $k:[a,b]\rightarrow
\mathbb{R}
$ be a continuous nonnegative map such that $k\left( t\right) ,$ $k^{\prime
}\left( t\right) \neq 0.\ $Then,
\end{theorem}

$i.\ \int\limits_{a}^{b}\left( f\left( x\right) +g\left( x\right) \right)
\frac{k^{\prime }\left( x\right) }{\left( k\left( x\right) \right)
^{1-\alpha }}dx=\int\limits_{a}^{b}\frac{f\left( x\right) k^{\prime }\left(
x\right) }{\left( k\left( x\right) \right) ^{1-\alpha }}dx+\int%
\limits_{a}^{b}\frac{g\left( x\right) k^{\prime }\left( x\right) }{\left(
k\left( x\right) \right) ^{1-\alpha }}dx,$

$ii.\ \int\limits_{a}^{b}\lambda \frac{f\left( x\right) k^{\prime }\left(
x\right) }{\left( k\left( x\right) \right) ^{1-\alpha }}dx=\lambda
\int\limits_{a}^{b}\frac{f\left( x\right) k^{\prime }\left( x\right) }{%
\left( k\left( x\right) \right) ^{1-\alpha }}dx,\ \lambda \in
\mathbb{R}
,$

$iii.\ \int\limits_{a}^{b}\frac{f\left( x\right) k^{\prime }\left( x\right)
}{\left( k\left( x\right) \right) ^{1-\alpha }}dx=-\int\limits_{b}^{a}\frac{%
f\left( x\right) k^{\prime }\left( x\right) }{\left( k\left( x\right)
\right) ^{1-\alpha }}dx,$

$iv.\ \int\limits_{a}^{b}\frac{f\left( x\right) k^{\prime }\left( x\right)
}{\left( k\left( x\right) \right) ^{1-\alpha }}dx=\int\limits_{a}^{c}\frac{%
f\left( x\right) k^{\prime }\left( x\right) }{\left( k\left( x\right)
\right) ^{1-\alpha }}dx+\int\limits_{c}^{b}\frac{f\left( x\right) k^{\prime
}\left( x\right) }{\left( k\left( x\right) \right) ^{1-\alpha }}dx,$

$v.\ \int\limits_{a}^{a}\frac{f\left( x\right) k^{\prime }\left( x\right) }{%
\left( k\left( x\right) \right) ^{1-\alpha }}dx=0,$

$vi.\ $if $f(x)\geq 0$ for all $x\in \lbrack a,b]$ , then $%
\int\limits_{a}^{b}\frac{f\left( x\right) k^{\prime }\left( x\right) }{%
\left( k\left( x\right) \right) ^{1-\alpha }}dx\geq 0,$

$vii.\ \left\vert \int\limits_{a}^{b}\frac{f\left( x\right) k^{\prime
}\left( x\right) }{\left( k\left( x\right) \right) ^{1-\alpha }}%
dx\right\vert \leq \int\limits_{a}^{b}\frac{\left\vert f\left( x\right)
\right\vert k^{\prime }\left( x\right) }{\left( k\left( x\right) \right)
^{1-\alpha }}dx.$

\begin{proof}
The relations follow from Definition \ref{d2} and Theorem \ref{T2},
analogous properties of generalized fractional integral, and the properties
of section 2 for the generalized fractional derivative.
\end{proof}

\begin{acknowledgement}
M.E. Yildirim was partially supported by the Scientific and Technological
Research Council of Turkey (TUBITAK Programme 2228-B).
\end{acknowledgement}


\begin{thebibliography}{99}
\bibitem{Akkurt} A. Akkurt, M.E. Y\i ld\i r\i m and H. Y\i ld\i r\i m, On
Some Integral Inequalities for Conformable Fractional Integrals, Asian
Journal of Mathematics and Computer Research, 15(3): 205-212, 2017.

\bibitem{Almeida} R. Almeida, M. Guzowska and T. Odzijewicz, A remark on
local fractional calculus and ordinary derivatives, arXiv preprint
arXiv:1612.00214.

\bibitem{Kilbas} A. Kilbas, H. Srivastava, J. Trujillo, Theory and
Applications of Fractional Differential Equations, in: Math. Studies.,
North-Holland, New York, 2006.

\bibitem{Katugampola} U. Katumgapola, A new fractional derivative with
classical properties, preprint.

\bibitem{Katugampola1} U.N. Katugampola, New Approach to a generalized
fractional integral, Appl. Math. Comput. 218(3), (2011), 860-865.

\bibitem{Abdel} T. Abdeljawad, On conformable fractional calculus, Journal
of Computational and Applied Mathematics 279 (2015) 57--66.

\bibitem{Khalil} R. Khalil, M. Al horani, A. Yousef, M. Sababheh, A new
definition of fractional derivative, Journal of Computational Apllied
Mathematics, 264 (2014), 65-70.

\bibitem{iyiola} O.S. Iyiola and E.R. Nwaeze, Some new results on the new
conformable fractional calculus with application using D'Alambert approach,
Progr. Fract. Differ. Appl., 2(2), 115-122, 2016.

\bibitem{Hammad} M. Abu Hammad, R. Khalil, Conformable fractional heat
differential equations, International Journal of Differential Equations and
Applications 13(3), 2014, 177-183.

\bibitem{Hammad1} M. Abu Hammad, R. Khalil, Abel's formula and wronskian for
conformable fractional differential equations, International Journal of
Differential Equations and Applications 13(3), 2014, 177-183.

\bibitem{Samko} Samko, S.G.; Kilbas, A.A.; Marichev, O.I.: Fractional
Integrals and Derivatives, Theory and Applications, Gordon and Breach,
Yverdon, Switzerland, 1993
\end{thebibliography}
\end{document}